\documentclass[11pt]{amsart}
\usepackage[all]{xy}
\usepackage{url}
\usepackage{amssymb}
\usepackage{hyperref}
\usepackage{enumerate}

\allowdisplaybreaks

\renewcommand{\AA}{\mathbb{A}}

\newcommand{\PP}{\mathbb{P}}
\newcommand{\QQ}{\mathbb{Q}}

\newcommand{\ZZ}{\mathbb{Z}}

\newcommand{\Div}{\operatorname{Div}}
\newcommand{\Gal}{\operatorname{Gal}}
\newcommand{\Gm}{\mathbb{G}_{\mathrm{m}}}
\newcommand{\Ecal}{\mathcal{E}}
\newcommand{\fp}{\mathfrak{p}}
\newcommand{\fq}{\mathfrak{q}}
\newcommand{\fP}{\mathfrak{P}}
\newcommand{\fQ}{\mathfrak{Q}}
\newcommand{\gp}{\mathfrak{p}}
\newcommand{\gq}{\mathfrak{q}}
\newcommand{\gP}{\mathfrak{P}}
\newcommand{\gQ}{\mathfrak{Q}}
\newcommand{\GL}{\operatorname{GL}}
\newcommand{\Kbar}{{\overline{K}}}

\newcommand{\MOD}[1]{~(\textup{mod}~#1)}
\newcommand{\Ocal}{\mathcal{O}}
\newcommand{\ptinf}{{O}}
\newcommand{\ord}{\operatorname{ord}}
\newcommand{\Spec}{\operatorname{Spec}}
\newcommand{\Supp}{\operatorname{Supp}}
\renewcommand{\div}{\operatorname{div}}

\renewcommand{\epsilon}{\varepsilon}
\newcommand{\g}{\gamma}
\newcommand{\s}{\sigma}
\newcommand{\hhat}{\widehat{h}}

\newtheorem{theorem}{Theorem}
\newtheorem{lemma}[theorem]{Lemma}

\newtheorem{prop}[theorem]{Proposition}
\theoremstyle{definition}
\newtheorem{definition}[theorem]{Definition}
\newtheorem{example}[theorem]{Example}
\newtheorem{remark}[theorem]{Remark}

\title[Algebraic divisibility sequences over function fields]{Algebraic divisibility sequences over function fields}
\author[Ingram, Mah\'{e}, Silverman, Stange, Streng]{Patrick~Ingram, Val\'{e}ry~Mah\'{e}, Joseph~H.~Silverman, Katherine~E.~Stange, and Marco~Streng}
\date{\today}
\address{Department of Mathematics, Colorado State University, Fort Collins, CO 80521, USA}
\email{pingram@math.colostate.edu}
\address{EPF Lausanne, SB-IMB-CSAG, Station 8, CH-1015 Lausanne, Switzerland}
\email{valery.mahe@epfl.ch}
\address{Mathematics Department, Brown University, Box 1917, Providence, RI 02912 USA}
\email{jhs@math.brown.edu}
\address{Mathematics Institute, University of Warwick, Coventry CV4 7AL, United Kingdom}
\email{marco.streng@gmail.com}
\address{Department of Mathematics, Stanford University, 450 Serra Mall, Building 380, Stanford, CA 94305 USA}
\email{stange@math.stanford.edu}

\thanks{
	Ingram's research is supported by a grant from NSERC of Canada.
     Mah\'e's research is supported by the universit\'e de Franche-Comt\'e.
     Silverman's research is supported by DMS-0854755.
     Stange's research has been supported by NSERC PDF-373333
           and NSF MSPRF 0802915.
     Streng's research is supported by EPSRC grant number EP/G004870/1}

\begin{document}

\begin{abstract}
In this note we study the existence of primes and of primitive
divisors in function field analogues of classical divisibility sequences.
Under various hypotheses, we prove that Lucas sequences and
elliptic divisibility sequences over function fields defined over
number fields contain infinitely many irreducible elements. We also
prove that an elliptic divisibility sequence over a function field has
only finitely many terms lacking a primitive divisor.
\end{abstract}

\maketitle

\begin{center}
\Large
\emph{In Memory of Alf van der Poorten,}
\par
\emph{Mathematician, Colleague, Friend}
\end{center}

\section{Introduction}

Integer sequences of the form
\begin{equation}
  \label{eqn:lucform}
  L_n = \frac{ f^n - g^n }{ f - g } \in \ZZ
\end{equation}
are called \emph{Lucas sequences} (of the first kind).  Necessarily,
$f$ and~$g$ are the roots of a monic quadratic polynomial~$p(x) \in
\ZZ[x]$.  The most famous examples are the Fibonacci numbers and the
Mersenne numbers, with~$p(x) = x^2 - x - 1$ and~$p(x) = (x-2)(x-1)$,
respectively.

Lucas sequences are associated to twisted forms of the multiplicative
group~$\Gm$.  Replacing~$\Gm$ with an elliptic curve yields an
analogous class of sequences.  Let~$E/\QQ$ be an elliptic curve given
by a Weierstrass equation,  let~$P \in E(\QQ)$ be a nontorsion
point, and write
\[
  x\bigl([n]P\bigr) = A_n / D_n^2 \in \QQ
\]
as a fraction in lowest terms. The integer sequence~$(D_n)_{n \ge 1}$ 
is called the \emph{elliptic divisibility sequence} (EDS)
associated to the pair~$(E,P)$.  Both Lucas sequences and EDS are
examples of \emph{divisibility sequences}, i.e.,
\[
  m \mid n \implies L_m \mid L_n\quad\text{and}\quad D_m\mid D_n.
\]

The primality of terms in integer sequences is an old question.  For
example, a long-standing conjecture says that the Mersenne sequence
$M_n=2^n-1$ contains infinitely many primes, and more generally it is
expected that a Lucas sequence will have infinitely many prime terms
\cite{MR1484896, MR1961589, MR2196797} unless it has a ``generic''
factorization \cite{MR2429645}. On the other hand, because of the rapid growth rate of EDS, which satisfy~$\log|D_n|\gg n^2$, the prime number theorem suggests that EDS should contain only finitely many primes~\cite{MR1815962}.

In this paper we study the problem of irreducible elements in Lucas
sequences and EDS defined over one-dimensional function fields~$K(C)$,
where $K$ is a number field.  We note that this is different from the
case of function fields over finite fields, where one would expect the
theory to be similar to the case of sequences defined over number
fields. We begin with a definition.

\begin{definition}
Let~$C/K$ be a curve defined over a number field~$K$. A divisor~$D\in
\Div(C_{\Kbar})$ is \emph{defined over~$K$} if it is fixed by
$\Gal(\Kbar/K)$. It is
\emph{semi-reduced}
if every point occurs with multiplicity $0$ or~$1$.

If $D$ is defined over~$K$ and semi-reduced, and~$\Gal(\Kbar/K)$
acts transitively on the support of~$D$, then we say that~$D$ is
\emph{irreducible over~$K$}.
\end{definition}

Let~$K$ be a number field.  We consider first Lucas sequences over the
coordinate ring~$K[C]$ of an affine curve~$C$. As we have noted, it is
not true that all Lucas sequences have infinitely many prime terms, so
we impose a technical restriction which we call \emph{amenability}. 
See Definition~\ref{def:amenable} in Section~\ref{sec:multmain} for
the full definition, but for example, amenable sequences include those
of the form
\[
  L_n = \frac{f(T)^n - 1}{f(T)-1} 
\]
with~$f(T)-1$ of prime degree and irreducible in the polynomial ring
$K[T]$. With the amenability hypothesis, we are able to prove
that~$L_q$ is irreducible for a set of primes~$q$ of positive lower
density (we recall the definition of Dirichlet density in Section~\ref{sec:multmain}).

\begin{theorem}
\label{th:mainmultresult}
Let~$K$ be a number field, let~$C/K$ be an affine curve, let~$K[C]$
denote the affine coordinate ring of~$C/K$, and let $L_n \in K[C]$
be an amenable Lucas sequence.  Then the set of primes~$q$ such that
$\div(L_q)$ is irreducible over~$K$ has positive lower
Dirichlet density.
\end{theorem}

\begin{example}
Let~$C$ be the affine line, so~$K[C]=K[T]$. Then a function~$f(T) \in
K[T]$ is irreducible if and only if its divisor~$\div(f) \in
\Div(C)$ is irreducible.  As a specific example,
the polynomial
\[
  L_q=\frac{(T^2+2)^q-1}{T^2+1} \in \QQ[T]
\] 
is irreducible in~$\QQ[T]$ for all primes~$q\equiv 3\MOD{4}$, although
we note that computations suggest that these~$L_q$ are in fact
irreducible for all primes~$q$.  See Section~\ref{sec:ex} for more
details on this example.
\end{example}

The definition of elliptic divisibility sequences over~$\QQ$ depends
on writing a fraction in lowest terms. We observe that the denominator
of the~$x$-coordinate of a point~$P$ on a Weierstrass curve measures
the primes at which~$P$ reduces to the point~$\ptinf$ at
infinity.  We use this idea in order to define
our more canonical
notion of EDS over function fields,
which does not depend on a choice of model, only
on $E/K$ and $P\in E(K)$.

\begin{definition} 
\label{def:ellsurfeds}
Let~$K(C)$ be the function field of a smooth projective curve~$C$,
let~$E/K(C)$ be an elliptic curve defined over the function field
of~$C$, and let~$\Ecal\to C$ be the minimal proper regular model
of~$E$ over $C$. \footnote{The minimal proper regular model
is a smooth projective surface over $K$
associated to~$E$. 
See Section~\ref{section:primitivedivisors} for more information.}
  Let~$\Ocal\subset \Ecal$ be the image of the zero
section. Each point~$P\in E\bigl(K(C)\bigr)$ induces a map
$\s_P:C\to\Ecal$. The \emph{elliptic divisibility sequence} associated
to the pair~$(E,P)$ is the sequence of divisors
\[
  D_{nP} = \s_{nP}^*(\Ocal) \in \Div(C),\qquad n\ge1.
\]
(If~$nP=\ptinf$, we leave~$D_{nP}$ undefined.)
\end{definition}

The general problem of irreducible elements in EDS over function
fields appears difficult. Even the case of a split elliptic curve,
which we study in our next result, presents challenges.

\begin{theorem}
\label{th:mainEDSresult} 
Let~$K$ be a number field, let~$K(C)$ be the function field of a curve
$C$, and let~$(D_{nP})_{n\ge1}$ be an elliptic divisibility sequence,
as described in Definition~$\ref{def:ellsurfeds}$, corresponding to a
pair~$(E,P)$. Suppose further that
\begin{enumerate}[ i)]
  \item the elliptic curve~$E$ is split, i.e.,~$E$ is isomorphic to a curve over~$K$;
  \item the elliptic curve~$E$ does not have CM;
  \item the point~$P\in
E\bigl(K(C)\bigr)$ is nonconstant; and
\item the divisor~$D_P$ is
irreducible over~$K$ and has prime degree.
\end{enumerate}
Then the set of rational
primes~$q$ such that the divisor \text{$D_{qP}-D_P$} is irreducible
has positive lower Dirichlet density.
\end{theorem}
\begin{remark}
If $P$ is constant, then the EDS is trivial. The condition
that $D_P$ is irreducible is also necessary, as counterexamples
can be obtained from Theorem~\ref{thm:isogenyEDS} below.
We will explain below Theorem~\ref{th:EDSprimitivedivisors} why
$q$ must be prime.

The other conditions, that $E$ is split and non-CM,
and that $D_P$ has prime degree, are consequences of our methods. 
We will use the Galois theory of $E[q]$ over~$K$, which 
looks very different if $E$ is non-split or CM. 
And we will employ the fact that $q$ is inert in the field 
extension $K(D_P)/K$ for a positive density of primes~$q$,
a fact that is true by Chebotarev's density theorem
if the degree of the field extension is prime
(Lemma~\ref{lem:EDSchebotarev}), but not in general.
\end{remark}

The proofs of Theorems~\ref{th:mainmultresult} and
\ref{th:mainEDSresult} are similar.  In both cases, the sequence in
question arises from a certain point~$P$ in an algebraic group
(the multiplicative group $\Gm$ in the former case) over~$K$. 
And in both cases, the point $P$ is defined over $\overline{K(C)}$, and the
$q^{\text{th}}$~term of the sequence corresponds to the divisor on $C$
over which the point $P$ meets the~$q$-torsion of the group.  If the
absolute Galois group of $K$ acts transitively
on the points of order~$q$,
then proving the irreducibility of the divisor is the same
as proving the irreducibility of the divisor of intersection of $P$
with a single $q$-torsion point.  We complete the proof by analyzing
the divisor locally at primes lying above~$q$.

Although the question of whether or not there are infinitely many
Mer\-senne primes is perhaps the best known problem concerning primes in
divisibility sequences, another question that has received a great
deal of attention in both the multiplicative and elliptic cases is the
existence of \emph{primitive divisors}.  A primitive divisor of a
term~$a_n$ in an integer sequence is a prime divisor of~$a_n$ that
divides no earlier term in the sequence.

Here we give a result for general one-dimensional function fields of
characteristic zero. We refer the reader to
Section~$\ref{section:primitivedivisors}$ for definitions and further
details, and to Section~$\ref{sec:history}$ for a discussion of work
on primitive divisors in other contexts.

\begin{theorem}
\label{th:EDSprimitivedivisors}
Let~$K$ be a field of characteristic zero, and let~$(D_{nP})_{n\ge1}$
be an EDS defined over~$K(C)$, the function field of a curve. Assume
further that there is no isomorphism
$\psi:E\rightarrow E'$ over $\Kbar(C)$ with $E'$ defined over $\Kbar$ and
$\psi(P)\in E'(\Kbar)$,
and assume
the point~$P$ is nontorsion. Then for all but finitely
many~$n$, the divisor~$D_{nP}$ has a primitive divisor.
\end{theorem}
\begin{remark}
The conditions on $E$ and $P$ in Theorem~\ref{th:EDSprimitivedivisors}
are necessary. Indeed, 
if an isomorphism $\psi$ as above
exists,
then the EDS is trivial, and if $P$ is torsion,
then it is periodic.
\end{remark}

Theorem \ref{th:mainEDSresult} focuses on the study of irreducible
terms~$D_{nP}$ in elliptic divisibility sequences over~$K(C)$ when the
index~$n$ is prime. The fact that~$D_{nP}$ is a divisibility sequence
suggests that this restriction to prime indices is necessary, since if
$m\mid n$, then~$D_{nP}$ always decomposes into a sum
$D_{nP}=D_{mP}+(D_{nP}-D_{mP})$ of divisors defined
over~$K$. Thus~$D_{nP}$ is reducible unless either~$D_{mP}=0$
or~$D_{nP}=D_{mP}$, and the theorem on primitive divisors
(Theorem~\ref{th:EDSprimitivedivisors}) says that~$D_{nP}\ne D_{mP}$
if~$n$ is sufficiently large.  More generally, a \emph{magnified EDS}
is an EDS that admits a type of generic factorization. We will prove
that magnified~EDS have only finitely many irreducible terms; see
Theorem~\ref{thm:two-primitive-divisors}, and
Theorem~\ref{thm:isogenyEDS} for a related stronger result.  We also
refer the reader to~\cite[Theorem~1.5]{MR2429645} for
effective bounds (for $K(C)=\QQ(t)$)
that are proven using the function field analogue of
the ABC~conjecture.

We conclude our introduction with a brief overview of the contents of
this paper. In Section~\ref{sec:history} we motivate our work with
some historical remarks on the study of primes and primitive divisors
in divisibility sequences.  Section~\ref{sec:multmain} gives the proof
of Theorem~\ref{th:mainmultresult} on the existence of irreducible
terms in Lucas sequences, and Section~\ref{section:irredtermsineds}
gives the proof of the analogous Theorem~\ref{th:mainEDSresult} for
(split) elliptic divisibility sequences. 
Section~\ref{section:primitivedivisors} contains the proof of
Theorem~\ref{th:EDSprimitivedivisors} on the existence of primitive
divisors in general~EDS over function fields. In Section~\ref{sec:mag}
we take up the question of magnification in EDS and use it to show
that a magnified EDS contains only finitely many irreducible terms. We
also briefly comment on the difficulties of extending our
irreducibility methods to non-isotrivial EDS. We conclude in
Section~\ref{sec:ex} with a number of examples illustrating
our results.

\subsection*{Acknowledgments}
This project was initiated at a conference at the 
International Centre for Mathematical Sciences
in Edinburgh in 2007 and originally included the five authors, Graham
Everest, and Nelson Stephens.  Graham is unfortunately no longer with
us, but his ideas suffuse this work, and we take this opportunity to
remember and appreciate his life as a valued colleague and friend.  We
also thank Nelson for his input during the original meeting, and
Maarten Derickx, Michael Rosen, and Jonathan Wise for helpful
discussions as the project approached completion.

\section{History and Motivation}
\label{sec:history}

In this section we briefly discuss some of the history of primes and
primitive divisors in divisibility sequences over various types of
rings and fields. This is primarily meant to provide background and to
help motivate our work over function fields.

The search for Mersenne primes~$2^n-1$ was initiated by the French monk
Marin Mersenne in the early~$17^{\text{th}}$-century and continues
today in the form of a distributed computer program currently running
on nearly half a million CPUs~\cite{GIMPS}.  More generally, most
integer Lucas sequences are expected to have infinitely many prime
terms \cite{MR1484896, MR1961589, MR2196797}.  The only obvious
exceptions occur with a type of generic factorization
\cite{MR2429645}. For example, if~$f$ and~$g$ are positive
coprime integers, then the Lucas sequence associated to~$f^2$ 
and~$g^2$, 
\begin{equation}
  \label{eqn:genfac}
   L_n = \frac{ f^{2n} - g^{2n} }{f^2 - g^2} =
   \left( \frac{f^n - g^n}{f - g} \right) \left(
   \frac{ f^n + g^n }{f + g} \right),
\end{equation}
contains only finitely many primes.

We remark that Seres~\cite{MR0155816, MR0179158} has considered
various irreducibility questions about compositions of the form
$\Phi_n\bigl(f(x)\bigr)$, where~$\Phi_n(x)$ is the~$n^{\text{th}}$
cyclotomic polynomial.  These results, however, all focus on the case
where \text{$f(x)\in\ZZ[x]$} has many integer roots,
while we focus on the case where $f-1$ is irreducible.

Elliptic divisibility sequences were first studied formally by Ward
\cite{MR0027286,MR0023275}, although Watson \cite{watson} considered
related sequences in his resolution of Lucas' square pyramid problem. 
Recently, the study of elliptic divisibility sequences has seen
renewed interest \cite{MR2220263,MR2178070,MR2747036,MR2226354,
  Stange10, MR2377368}, including applications to Hilbert's 10th
problem \cite{MR2377127,MR2480276,MR1992832} and cryptography
\cite{StangeLauter09, Shipsey00, MR2423649}.  (We remark that some
authors use a slightly different definition of EDS via the division
polynomial recursion. See the cited references for details.
These definitions differ only in finitely many valuations
(see \cite[Th\'eor\`eme~A]{MR1185022}).)

The~$n^{\text{th}}$~Mersenne number~$M_n$ can be prime only if~$n$ is
prime, and the prime number theorem suggests that~$M_q$ has
probability~$1/\log M_q$ of being prime. Thus the number of prime
terms~$M_q$ with~$q\leq X$ should grow like~$\sum_{q\le X}
q^{-1}\approx\log\log(X)$.  This argument fails to take into account some
nuances, but a more careful heuristic analysis by Wagstaff
\cite{MR679454} refines this argument and gives reason to believe that
the number of~$q\leq X$ such that~$M_q$ is prime should be asymptotic
to~$e^\gamma\log\log_2(X)$.

The study of prime terms of elliptic divisibility sequences began with
Chudnovsky and Chudnovsky~\cite{MR0866702}, who searched for primes
computationally.  An EDS over~$\ZZ$ grows much faster:~$\log |D_n|\gg {n^2}$,
and again only prime indices can give prime terms (with finitely many
exceptions), so a reasonable guess is that
\[
  \#\{n\ge1 : \text{$D_n$ is prime}\} 
  \ll \sum_{\text{$q$ prime}} \frac{1}{\log D_q}
  \ll \sum_{\text{$q$ prime}} \frac{1}{q^2}
  \ll 1.
\]
Building on the heuristic argument above, Einsiedler, Everest, and Ward~\cite{MR1815962} conjectured that an EDS
has only finitely many prime terms, and this conjecture was later
expanded upon by Everest, Ingram, Mah\'{e} and Stevens
\cite{MR2429645}.  For some EDS, finiteness follows from a type of
generic factorization not unlike \eqref{eqn:genfac}
(cf.\ \cite{MR2429645, MR2164113, MR2045409, Mahe-Explicit-bounds} and Section
\ref{sec:mag}), but the general case appears difficult.

The study of primitive divisors in integral Lucas sequences goes back
to the~$19^{\text{th}}$-century work of Bang \cite{Bang} and Zsigmondy
\cite{MR1546236}, who showed that~\text{$a^n-b^n$} has a primitive divisor
for all~$n>6$, and has a long history
\cite{MR1502458,MR0344221,MR0476628,MR1284673}, culminating in the
work of Bilu, Hanrot, and Voutier \cite{MR1863855}, who proved that a
Lucas sequence has primitive divisors for each index~$n>30$.  Flatters
and Ward considered the analogous question over polynomial
rings~\cite{flattersward}.

Work on primitive divisors in EDS is more recent, although we note
that in~1986 the third author included the existence of primitive
divisors in EDS as an exercise in the first edition
of~\cite{MR2514094} (for the full proof, see~\cite{MR961918}). 
A number of authors have given bounds on the
number of terms and/or the largest term that have no primitive divisor
for various types of EDS, as well as studying generalized primitive
divisors when~$\operatorname{End}(E)\ne\ZZ$; see~\cite{MR2220263,
  MR2301226, MR2605536, ingramsilverman06, MR2377368,
  voutieryabuta}.  The proofs of such results generally require deep
quantitative and/or effective versions of Siegel's theorem on
integrality of points on elliptic curves.

\section{Proof of Theorem~\ref{th:mainmultresult}---Irreducible Terms in Lucas Sequences}
\label{sec:multmain}

For this section, we let~$K$ be a number field, we take~$C/K$ to be a
smooth affine curve defined over~$K$, and we write~$K[C]$ for the
affine coordinate ring of~$C/K$.  We begin with the definition of
amenability, after which we prove that amenable Lucas sequences over
$K[C]$ have infinitely many irreducible terms.

\begin{definition}
The degree of a divisor 
\[
  D=\sum_{P\in C} n_P (P) \in\Div(C_{\Kbar})
  \quad\text{is the sum}\quad
  \deg(D)=\sum_{P \in C} n_P.
\]
For a regular function~$f\in K[C]$, we write~$\deg(f)$ for the degree
of the divisor of zeros of~$f$, i.e.,
\[
  \deg(f) = \sum_{P\in C} \ord_P(f).
\]
We note that since~$C$ is affine, there may be some zeros of~$f$ ``at
infinity'' that aren't counted. It need not be true that
$\deg(f~+~g)~\le~\max\bigl\{\deg(f),\deg(g)\bigr\}$.
\end{definition}

We are now ready to define our notion of amenability.

\begin{definition}
\label{def:amenable}
Let
\[
  L_n = \frac{f^n - g^n}{f - g} \in K[C]
\]
be a Lucas sequence.  First assume $f,g\in K[C]$. 
We then say that the sequence is
\emph{amenable \textup{(}over $K[C]$\textup{)}}
if the following three conditions hold:
\begin{enumerate} 
  \item~$\div(f-g)$ is irreducible over~$K$ and of prime degree, 
  \item~$\deg(f-g)$ is the generic degree of~$a f + b g$ as~$a, b$ range through~$K$,
  \item~$f$ and~$g$ have no common zeroes.
\end{enumerate}

In general, $f$ and $g$ are the roots
of the quadratic polynomial 
$$X^2 - L_2 X + (L_2^2-L_3)$$ over~$K[C].$ 
Let $C'\rightarrow C$ be a cover such that
$K[C']$ is the integral closure of $K[C]$ in
the field extension $K(C,f,g)/K(C)$. 
Now we have $f,g\in K[C']$ and
either $C'$ equals $C$, or $C'\rightarrow C$ is a double cover. 
We call $L_n$ \emph{amenable
\textup{(}over $K[C]$\textup{)}
} if it is amenable over~$K[C']$.
\end{definition}

\begin{example} 
\label{rem:lucaskt}
Suppose that we are in the case $C=\AA^1$,
i.e., $L_n$ is a Lucas sequence in the polynomial ring~$K[T]$. 
There are two cases. First, if~$f$ and~$g$ are
themselves in~$K[T]$, then~$(L_n)_{n\ge1}$ is amenable if and only if
\begin{enumerate}
  \item~$f - g$ is an irreducible polynomial of~$K[T]$ of prime degree,
  \item~$\deg(f - g) = \max\{ \deg(f), \deg(g) \}$,
  \item~$f$ is not a constant multiple of~$g$.
\end{enumerate}
Second, if~$f$ and~$g$ are quadratic over~$K[T]$,
then they are conjugate, and both
\text{$f + g$} and \text{$(f - g)^2$} are in
$K[T]$. In this case, the sequence is amenable if and only if
\begin{enumerate}
  \item~$(f - g)^2$ is an irreducible polynomial of~$K[T]$ of
    prime degree,
  \item~$\deg(f + g)\leq \frac{1}{2}\deg\bigl((f - g)^2\bigr)$,
  \item~$f + g \neq 0$.
 \end{enumerate}
\end{example}

The following lemma provides the key tool in the proof of
Theorem~\ref{th:mainmultresult}.

\begin{lemma}
  \label{lem:multlemma}
Let~$f, g\in K[C]$ be such that the associated Lucas sequence 
\[
  L_n = \frac{f^n-g^n}{f-g}
\]
is amenable, let
\[
  D_0 = \div(f-g),
\]
and define two sets of primes by
\begin{align*}
  S&=\left\{\fq\subset \Ocal_K\text{ prime } :
  \begin{tabular}{@{}l@{}}
    there is a rational prime~$q$ such that\\
   ~$\fq\mid q$ and~$\div(L_q)$ is irreducible over~$K$\\
  \end{tabular}
  \right\}, \\
  M&=\left\{\fq\subset \Ocal_K\text{ prime } :
   \begin{tabular}{@{}l@{}}
   $C$ is smooth over $(\Ocal_K/\fq)$ and\\
   $D_0$ is irreducible over~$(\Ocal_K/\fq)$
   \end{tabular}
  \right\}.
\end{align*}
Then there is a finite set~$S'$ of primes of~$\Ocal_K$ such that
\[
  M\subseteq S\cup S'.
\]
\end{lemma}

\begin{proof}
Let~$q$ be a prime, and let~$\zeta$ be a primitive~$q^{th}$-root of
unity.  Working in~$K(\zeta)[C]$, the function~$L_q$ factors as
\begin{equation}
  \label{eqn:lq-prod}
   L_q=\frac{f^q-g^q}{f-g}=\prod_{j=1}^{q-1}\left(f-\zeta^j g\right).
\end{equation}
Define the corresponding divisors on~$C$ by
\[
  D_j = \div(f - \zeta^j g)\quad\text{for~$0 \le j \le q-1$.}
\]
We claim that the divisors~$D_0, \ldots D_{q-1}$ have pairwise
disjoint support.  To see this, suppose that~$P\in C(\Kbar)$ is a
common zero of~$f - \zeta^ig$ and~$f- \zeta^j g$ for some~$i \neq j$. 
Then~$P$ is a common zero of~$f$ and~$g$, which contradicts
Property~(3) of amenability.

We now assume that~$q$ is chosen sufficiently large so that
$q$ is unramified in~$K$. 
This implies that
$\QQ(\zeta)$ is linearly disjoint from~$K$ over~$\QQ$
(because~$q$ is totally ramified in $\QQ(\zeta)$ and unramified in~$K$). 
Then the group
$\Gal(K(\zeta)/K)\cong\Gal(\QQ(\zeta)/\QQ)$ acts transitively on the
terms in the product~\eqref{eqn:lq-prod}, so it also acts transitively
on the divisors~$D_j$ with~$1\le j\le q-1$.  Thus, in order to show
that
\[
  \div(L_q) = \sum_{j=1}^{q-1} D_j
\]
is irreducible over~$K$, it suffices to show that~$D_j$ is irreducible
over~$K(\zeta)$ for some~$1 \le j \le q-1$.  We do this by showing
that the reduction~$\widetilde{D_j}$ modulo some prime of~$K(\zeta)$
is irreducible and has the same degree as~$D_j$.
 
Choose primes~$\fQ\subseteq \Ocal_{K(\zeta)}$ and~$\fq\subseteq
\Ocal_K$ with~$\fQ\mid\fq\mid q$.  We may suppose that~$q$ is taken
large enough so that the reductions of~$f$ and~$g$ modulo~$\fQ$, which
we denote by~$\tilde{f}, \tilde{g} \in k_\fQ[\tilde{C}]$, are
well-defined and satisfy
\[
  \deg\tilde{f}=\deg f
  \quad\text{and}\quad
  \deg\tilde{g}=\deg g.
\]
(Here~$k_\fQ$ denotes the residue field of~$\Ocal_{K(\zeta)}$ at~$\fQ$.)

In general, there may be a finite set of rational primes~$q$ such that
some point~$P \in \Supp(D_j)$ reduces modulo~$\fQ$ to a point not on
the affine curve~$C$.  If this happens, then
\[
  \deg(\widetilde{D_j}) < \deg(D_j).
\]
We wish to rule out this possibility.  For~$D_0$, which does not
depend on~$q$, it suffices to assume that~$q$ is sufficiently large. 
For~$D_j$, we compare the degree before and after reduction.

Let~$d = \deg(D_0)$ over~$K[C]$.  By part~(2) of the amenability hypothesis over
$K[C]$, we have
\[
  \deg(D_j) \le d = \deg(D_0).
\]
Further, since~$1-\zeta^j \in \fQ$, we see that
\begin{equation}
  \label{eqn:fzg-reduce}
  f-\zeta^j g\equiv f-g\MOD{\fQ}.
\end{equation}
Hence~$\widetilde{D_j} = \widetilde{D_0}$, and the degree of~$D_j$ is
$d$ both before and after reduction modulo~$\fQ$.

We now assume that~$\fq \in M$, so that~$\widetilde{D_0}\bmod\fq$ is
irreducible over~$k_\fq$.  Since~$K(\zeta)/K$ is totally ramified
at~$\fq$, the residue fields
\[
  k_\fQ = \Ocal_{K(\zeta)} / \fQ
  \quad\text{and}\quad
  k_\fq=\Ocal_K/\fq
  \quad\text{are equal,}
\]
and hence~$\widetilde{D_j} = \widetilde{D_0}$ is irreducible over this
finite field.  The degrees of $D_j$ and $\widetilde{D_j}$ being equal,
it follows that~$D_j$ is irreducible over~$K$, and so
$\div(L_q)$ is irreducible over~$K(\zeta)$.  Since we have excluded only a
finite number of primes, this proves the lemma.
\end{proof}

\begin{definition}
Let $K$ be a number field and $P_K$ its set of primes. 
The \emph{Dirichlet density} of a subset $M\subset P_K$
is defined as
$$d(M) = \lim_{s\downarrow 1} \frac{\sum_{\gp\in M} N(\gp)^{-s}}{\sum_{\gp\in P_K} N(\gp)^{-s}},$$
if that limit exists. 
We define the \emph{lower} Dirichlet density $d_-(M)$ by taking $\liminf$ instead of~$\lim$.
\end{definition}
We will relate densities
of sets of primes of $K$ and $\QQ$ as follows.
\begin{lemma}\label{lem:relatedensities}
Let $K$ be a number field and $M_K\subset P_K$. 
Let $M_{\QQ}=\{p\in P_{\QQ} \mid \exists \gp\in M: N(\gp)=p\}$. 
Then we have $$d_{-}(M_{\QQ})\geq \frac{d_{-}(M_{K})}{[K:\QQ ]} $$
\end{lemma}
\begin{proof}
It is shown in \cite[\S~13]{neukirch} that
the limit defining $d_{-}(M_{K})$
does not change if we remove from $M_K$ all primes of degree~$>1$,
and replace the denominator by $\log(1/(s-1))$. 
So assume without loss of generality
that $M_K$ contains only primes of degree~$1$. 
For every element of $M_{\QQ}$, there are at
most $[K :\QQ ]$ elements of $M_{K}$, hence
we get $$d_{-}(M_{\QQ})=
\lim_{s\downarrow 1} \frac{\sum_{p\in M_{\QQ}} p^{-s}}{\log\frac{1}{s-1}}
\geq \frac{1}{[K :\QQ ]}
\lim_{s\downarrow 1} \frac{\sum_{\gp\in M_{K}} N(\gp)^{-s}}{\log\frac{1}{s-1}}
= \frac{d_{-}(M_{K})}{[K :\QQ ]}.$$
\end{proof}
We will need the following easy consequence of the Chebotarev
Density Theorem.
\begin{lemma}
\label{lem:EDSchebotarev}
Let~$D$ be a divisor of prime degree defined over~$K$ such that~$D$ is
irreducible over~$K$. Then there is a set~$T$ of primes of~$K$ of
positive density such that~$\widetilde{D}\bmod\fq$ is irreducible over
$k_\fq$ for all~$\fq\in T$.
\end{lemma}

\begin{proof}
Let~$p=\deg(D)$, which by assumption is prime.  By excluding a finite
set of primes, we may suppose that~$C$ has good reduction at every
$\fq$ under consideration.  Let~$L/K$ be the Galois extension of~$K$
generated by the points in the support of~$D$. If~$Q\in \Supp(D)$ is
any point, then the irreducibility of~$D$ over~$K$ implies that
$[K(Q):K]=p$, so~$p\mid \#\Gal(L/K)$.  It follows that the set
$X\subseteq\Gal(L/K)$ of elements acting as a~$p$-cycle on the support
of~$D$ is non-empty, and this set is conjugacy-invariant.  By the
Chebotarev Density Theorem (\cite[Theorem~13.4]{neukirch}),
there is a set of primes~$T$ of~$K$ of
density~$\#X/[L:K]$ such that for~$\fQ\mid \fq\in T$, the Frobenius
element of~$\Gal(k_\fQ/k_\fq)$ acts as a~$p$-cycle on the support of
the reduction of~$D$ modulo~$\fQ$.  In particular, for these~$\fq$ the
reduction of~$D$ modulo~$\fq$ is irreducible over~$k_\fq$.
\end{proof}

We now have the tools needed to prove that amenable Lucas sequences
over~$K[C]$ contain a significant number of irreducible terms.

\begin{proof}[Proof of Theorem~$\ref{th:mainmultresult}$]
Write $L_n = (f^n - g^n)/(f-g)$. Assume first $f,g\in K[C]$. 
By Lemma~\ref{lem:relatedensities}
it suffices to prove that the set~$S$
of Lemma~\ref{lem:multlemma} has positive lower density. Since the
set~$S'$ in Lemma~\ref{lem:multlemma} is finite, it suffices to prove
that the set~$M$ in Lemma~\ref{lem:multlemma} has positive lower
density.  But this follows from the amenability assumption and
Lemma~\ref{lem:EDSchebotarev},
which finishes the proof in case $f,g\in K[C]$.

In general, let $c:C'\rightarrow C$ be as in the definition of amenable. Then
We find that there is a set of primes $q$ of positive lower density
such that $c^* \div(L_q) = \div(L_q\circ c)\in \Div[C'](K)$ is irreducible. 
This implies that $\div(L_q)\in \Div[C](K)$ is irreducible as well.
\end{proof}

\section{Proof of Theorem~\ref{th:mainEDSresult}---Irreducible Terms in EDS}
\label{section:irredtermsineds}

Recall that Theorem~\ref{th:mainEDSresult} assumes that the elliptic
curve $E$ is defined over~$K$. 
We postpone the general
definition of the minimal proper regular model to Section~\ref{section:primitivedivisors},
and for now claim that if $E$ is defined over $K$, then
its minimal proper regular model is $\Ecal= E\times C$. 
Note that a point $Q\in E(K(C))$ induces a map $C\rightarrow E$
that by abuse of notation we denote
by~$\sigma_Q$. The map
$\sigma_Q:C\rightarrow \Ecal$ from the introduction
is now given by $\sigma_Q = (\sigma_Q \times \mathrm{id}_C)$. 
As a consequence, the EDS associated to $P$ is simply given
by $$D_{nP} = \sigma^*_{nP}(\ptinf)\in\mathrm{Div}(C),\quad n\geq 1,$$
and we will not use $\Ecal$
in this section.

The proof of Theorem~\ref{th:mainEDSresult} proceeds along similar
lines to the proof of Theorem~\ref{th:mainmultresult}, but the proof
is complicated by the fact that there are no totally ramified primes,
so we must use another argument to find appropriate primes of
degree~$1$.  We begin with the key lemma, which is used in place of
the fact that~$q^{\text{th}}$-roots of unity generate totally ramified
extensions.

\begin{lemma}
\label{lemma:deg1prime}
Let~$E/K$ be an elliptic curve defined over a number field, and assume
that~$E$ does not have CM. Then for all prime ideals~$\gp$
of~$K$ such that~$p=N_{K/\QQ}(\gp)$ is prime and sufficiently large and such
that~$E$ has ordinary reduction at~$\gp$, and for all
points~\text{$Q\in E[p]$}, there exists a degree~$1$ prime ideal~$\fP
\mid \gp$ of the field~$K(Q)$ such that
\text{$Q\equiv \ptinf\MOD{\fP}$}.
\end{lemma}
\begin{proof}
Given~$E/K$, for all sufficiently large primes~$p$, the
following conditions hold:
\begin{itemize}
\setlength{\itemsep}{0pt}
\item $p$ is unramified in~$K$,
\item
$E$ has good reduction at all primes lying over~$p$, and
\item
the Galois group~$\Gal\left(K\bigl(E[p]\bigr)/K\right)$ acts
transitively on~$E[p]$. 
\end{itemize}
It is clear that the first two conditions
eliminate only finitely many
primes, and Serre's theorem~\cite{MR0387283} says that the
same is true for the third, since we have assumed
that~$E$ does not have~CM. 
\par
Let $\gp$ and $Q$ be as in the lemma. 
To ease notation, let~$L=K\bigl(E[p]\bigr)$ and~$L'=K(Q)$. 
Let~$\gP_0$ be a prime of~$L$ lying over~$\gp$. The
reduction-mod-$\gP_0$ map is not injective
on~$p$-torsion~\cite[III.6.4]{MR2514094}, so we can find a nonzero
point~\text{$Q_0\in E[p]$} such that \text{$Q_0\equiv \ptinf\MOD{\gP_0}$}. 
Since~$\Gal(L/K)$ acts transitively on~$E[p]$, we can find
a~\text{$g\in\Gal(L/K)$} such that~$g(Q_0)=Q$. Then setting
\[
  \gP=g(\gP_0),\quad \text{and}\quad \gP'=\gP\cap L',
\]
we have 
\[
  \gp=\gP'\cap K,
  \quad \text{and} \quad  Q\equiv \ptinf\pmod{\gP'}.
\]
For the convenience of the reader, the following display shows the
fields and primes that we are using:
\[
  \begin{array}{c@{\qquad}c}
   L=K\bigl(E[p]\bigr) & \fP \\
   \big| & \big| \\
   L'=K(Q) & \fP' \\
   \big| & \big| \\
   K & \fp \\
   \big| & \big| \\
  \QQ & p \\
  \end{array}
\]
It remains to prove that~$\gP'$ is a prime of degree~$1$.  

Since we have assumed that~$\gp$ has degree~$1$ over~$\QQ$, it suffices to
prove that the extension of residue fields~$k_{\gP'}/k_{\gp}$ is
trivial. This is done using ramification theory.  We denote by
$D_{\gP}$ and~$I_{\gP}$, respectively, the decomposition group and the
inertia group of~$\gP$.  
The degree of~$\gP'$ is~$1$ exactly when
$D_{\gP}\subset I_{\gP}\Gal\left( L/L'\right)$.
We prove this inclusion of sets using Serre's results~\cite{MR0387283}
that describe the~$\Gal\left(L_{\gP}/K_{\mathfrak{p}}\right)$-module structure of
$E[p]$, where~$L_{\gP}$,~$L'_{\gP'}$, and~$K_{\mathfrak{p}}$ denote the completions
of~$L$,~$L'$, and~$K$, respectively.

Let 
\[
  \rho_{p}: \Gal\left( L/K\right)\longrightarrow\GL\left(E[p]\right)
\]
be the Galois representation associated to~$E[p]$. 
Recall that $E$ has ordinary reduction at~$\gp$
and that $p$ is unramified in~$K/\QQ$,
so Serre~\cite[\S$1.11$]{MR0387283} shows the
existence of a basis~$\left(Q_{1},Q_{2}\right)$ of~$E[p]$ with
$Q_{1}\equiv \ptinf\pmod{\gP'}$, and such that under the isomorphism
$\GL\left(E[p]\right)\cong\GL_{2}\left(\mathbb{F}_{p}\right)$
associated to the basis~$\left(Q_{1},Q_{2}\right)$, the following two
facts are true:
\begin{itemize}
\item 
The image of~$D_{\gP}$ under~$\rho_{p}$ is contained in the Borel
subgroup~$\left\{\left(\begin{smallmatrix}
  *&*\\0&*\\\end{smallmatrix}\right)\right\}$ of
 ~$\GL_{2}\left(\mathbb{F}_{p}\right)$.
\item 
The image of~$I_{\gP}$ under~$\rho_{p}$ contains the subgroup
~$\left\{\left(\begin{smallmatrix} *&0\\0&1\\\end{smallmatrix}\right)\right\}$
of order \text{$p-1$}.
\end{itemize}

Under our assumption that~$E$ has ordinary reduction, the kernel of
reduction modulo~$\gP$ is cyclic of order~$p$, so~$Q_1$ is a 
multiple of the point~$Q$. Hence $\Gal(L/L')$ is the subgroup of~$\Gal(L/K)$ consisting in automorphisms acting trivially on $Q_1$. Since $\Gal (L/K)$ acts transitively on~$E[p]$, the image of~$\rho_{p}$ satisfies
\[
  \rho_{p}\bigl(\Gal(L/L')\bigr)
  =\left\{\begin{pmatrix} 1&*\\0&*\\ \end{pmatrix} \right\}
  \subset \GL_{2}\left(\mathbb{F}_{p}\right).
\]
In particular,~$\rho_{p}\bigl(\Gal(L/L')\bigr)$ has order~$p(p-1)$. 
It follows that
\[
  \rho_{p}\left(D_{\gP}\right)
  \subset\rho_{p}\bigl( I_{\gP}\Gal\left(L/L'\right)\bigr).
\]
\end{proof}

We next use Lemma~\ref{lemma:deg1prime} to prove an elliptic curve
analogue of Lemma~\ref{lem:multlemma}.

\begin{lemma}
\label{lem:EDSlemmaM_P}
Let~$E$,~$P$, and~$K$ be as in the statement of
Theorem~$\ref{th:mainEDSresult}$, and define sets of primes
\begin{align*}
  U_E &= \left\{\fq\subset \Ocal_K\text{ prime } :
  \begin{tabular}{@{}l@{}}
   ~$N_{K/\QQ}(\gq)$ is prime, i.e.,~$\gq$ has degree~$1$, \\
    and~$E$ has good ordinary reduction at~$\fq$\\
  \end{tabular}
  \right\}, \\*
  S_P &= \bigl\{ \gq\in U_E : \text{$D_{qP}-D_P$ is irreducible over~$K$,
          where~$q=N_{K/\QQ}(\gq)$} 
  \bigr\}, \\*
  M_P &=  \bigl\{ \gq\in U_E : 
    \text{$D_P$ modulo~$\fq$ is irreducible over~$\Ocal_K/\fq$}
  \bigr\}.
\end{align*}
Then there is a finite set~$S'$ of primes of~$\Ocal_K$ such
that 
\[
  M_P\subseteq S_P\cup S'.
\]
\end{lemma}
\begin{proof}
The point~$P\in E\bigl(K(C)\bigr)$ induces a map~$\s_P:C\to E$, and
our assumption that~$P$ is not a constant point, i.e.,~$P\notin E(K)$,
implies that~$\s_P$ is a finite covering. For any rational prime~$q$,
we have
\begin{equation}
  \label{eqn:DqPDPsum}
  D_{qP}-D_P 
  = \s_{qP}^*(\ptinf) - \s_P^*(\ptinf) 
  = \sum_{Q\in E[q]\smallsetminus\{\ptinf\}} \s_P^*(Q).
\end{equation}
\par
As noted in the proof of
Lemma~\ref{lemma:deg1prime}, if~$q$ is sufficiently large, then~$\Gal(\Kbar/K)$ acts
transitively on~$E[q]\smallsetminus\{\ptinf\}$. Thus~$\Gal(\Kbar/K)$ acts
transitively on the summands in the right side of~\eqref{eqn:DqPDPsum}, so in order to
prove that~\text{$D_{qP}-D_P$} is irreducible over~$K$, it suffices to
take a nonzero point~$Q\in E[q]$ and show that~$\s_P^*(Q)$ is
irreducible over~$L':=K(Q)$. 
\par
Let~$\gq\in M_P$, so in particular~$\gq$ has degree~$1$, and let
$q=N_{K/\QQ}(\gq)$. We want to show that~$\gq\in S_P$ (if~$q$ is
sufficiently large.) We will do this by finding a prime~$\gQ$ in~$L'$
such that~$\s_P^*(Q)\bmod\gQ$ is irreducible over the finite
field~$\Ocal_{L'}/\gQ$.  (This suffices, since the reduction
modulo~$\gQ$ of a reducible divisor is clearly reducible.)
\par
Lemma~\ref{lemma:deg1prime} says that if~$q$ is
sufficiently large, then there is a prime~$\gQ$ in~$L'$ of degree~$1$
over~$\gq$ such that~$Q\equiv \ptinf\MOD{\gQ}$. Thus
\[
  \s_P^*(Q) \equiv \s_P^*(\ptinf) \pmod{\gQ},
\]
so it suffices to prove that~$\s_P^*(\ptinf)\bmod\gQ$ is irreducible
over~$\Ocal_{L'}/\gQ$. 
\par
We have assumed that~$\gp\in M_P$, so by the definition of~$M_P$, we
know that~$D_P\bmod\gq$ is irreducible over the finite
field~$\Ocal_K/\gq$. Since further~$\gQ$ has degree~$1$ over~$\gq$,
this implies that~$D_P\bmod\gQ$ is irreducible over~$\Ocal_{L'}/\gQ$,
which completes the proof of the lemma.
\end{proof}

We now have the tools to complete the proof.

\begin{proof}[Proof of Theorem~$\ref{th:mainEDSresult}$]
We continue with the notation in the statement of
Lemma~\ref{lem:EDSlemmaM_P}.  We recall
from Lemma~\ref{lem:relatedensities} that if~$T$ is a set of primes
of~$K$ having positive lower density, then the set of rational primes
divisible by elements of~$T$ has positive lower density in the primes
of~$\QQ$.  So in order to prove Theorem~\ref{th:mainEDSresult}, it
suffices to prove that the set~$S_P$ has positive lower density. Since
the set~$S'$ in Lemma~\ref{lem:EDSlemmaM_P} is finite, it suffices to
prove that the set~$M_P$ in Lemma~\ref{lem:EDSlemmaM_P} has positive
lower density. 
\par
We are assuming that the divisor~$D_P$ is irreducible over~$K$ and has
prime degree. 
By Lemma~\ref{lem:EDSchebotarev},
the divisor
$D_P$ modulo~$\gq$ is irreducible for a set of primes of positive
density;
and since
the primes where~$E$ has supersingular reduction have density
zero~\cite{MR1144318,MR644559,MR1484415}, the same is true if we
restrict to primes where~$E$ has ordinary reduction.  This proves
that~$M_P$ has positive lower density, which completes the proof of
Theorem~$\ref{th:mainEDSresult}$.
\end{proof}

If~$D_P$ is reducible, then the conclusion of
Theorem~\ref{th:mainEDSresult} may be false. 
A counterexample is the case that~$C$ is an elliptic curve and the section $\sigma_P: C \rightarrow E$ is an isogeny of degree at least~$2$. 
Notice that in this case, the divisor~$D_P$ is never irreducible,
because its support contains~$\ptinf_C$, the zero point of~$C$. 
The same holds for $D_{qP}-D_{P}$, as its support
contains $\sigma^*_P(\ptinf_E)$. 
However, if we remove this divisor $\sigma^*_P(\ptinf_E)$,
then under a mild hypothesis,
we can prove that the remaining divisor $\sigma^*_P[q]^*(\ptinf_E)-\sigma^*_P(\ptinf_E)$
\emph{is} irreducible for \emph{almost all} primes~$q$,
not just a positive density. This is the following theorem.

\begin{theorem}
\label{thm:isogenyEDS}
We continue with the notation of the statement and proof of
Theorem~$\ref{th:mainEDSresult}$. Suppose that~$C$ is an elliptic curve
isogenous to~$E$ and \text{$\sigma_P: C \rightarrow E$} is an isogeny of
degree~$d > 1$.  Further, assume that~$\Gal(\Kbar/K)$ acts
transitively on~$\ker\left(\sigma_{P} \right)\smallsetminus \{ \ptinf_E
\} .$ Then for all sufficiently large rational primes~$q$, the divisor
$D_{qP}-D_P$ is a sum of exactly two irreducible divisors,
one of degree $(d-1)(q^2-1)$ and one of degree $q^2-1$.
\end{theorem}

\begin{proof}
Let~$q$ be a rational prime
with~\text{$q\nmid d$}.  Then 
\[
  D_{qP} - D_P 
  = \sigma_P^*[q]^*(\ptinf_E) - \sigma_P^*(\ptinf_E) 
  =\sum_{Q\in\ker\left(\sigma_P \circ [q]\right)\smallsetminus\ker\left(\sigma_P\right)} (Q).
\]
The decomposition of~$D_{qP} - D_P$ into a sum of irreducible divisors
over~$K$ will follow from the decomposition of~$\ker\left(\sigma_P \circ
[q]\right) = \ker\left(\sigma_{P}\right)\oplus C[q]$ into a union of
orbits under the action of~$\Gal(\Kbar/K)$.

To ease notation, we let $L=K\bigl(\ker(\s_P)\bigr)$.  As remarked at
the beginning of the proof of Lemma~\ref{lemma:deg1prime}, Serre's
theorem~\cite{MR0387283} implies that if~$q$ is sufficiently large,
then $\Gal(\Kbar/L)$ acts transitively on the set~$C[q]\smallsetminus
\{ \ptinf_C \}$.  (Note that we are assuming that~$E$ does not have~CM,
so the same holds for the isogenous elliptic curve~$C$.)  Further, we
have assumed that~$\Gal\left(\Kbar /K\right)$ acts transitively
on~$\ker\left(\sigma_P\right)\smallsetminus \{ \ptinf_C \}$.  Therefore
the set $\ker\left(\sigma_{P}\right)\oplus C[q]$ decomposes into the
following four Galois orbits:
\begin{itemize}
\item[(i)]
    $\{(\ptinf_C,\ptinf_C) \}$, 
\item[(ii)]
    $\{(R,\ptinf_C) : R\in\ker(\sigma_P), R\ne\ptinf_C\}$,
\item[(iii)]
    $\{(\ptinf_C,S) : S\in C[q], S\ne\ptinf_C\}$,
\item[(iv)]
    $\{(R,S) : 
     R\in\ker(\sigma_P), S\in C[q], R\ne\ptinf_C~\text{and}~S\ne\ptinf_C\}$.
\end{itemize}
Since~$D_{qP}-D_P$ consists of orbits~(iii) and~(iv),
which have the correct cardinalities,
this concludes
the proof.
\end{proof}

\begin{remark}
Theorem \ref{thm:isogenyEDS} gives a factorization of a division
polynomial associated to a composition of isogenies.  In the general
case, the same proof can be used to deduce for~$q$ large enough a
decomposition of~$D_{qP} - D_P$ into a sum of irreducible divisors
over~$K$ from a decomposition of~$\ker\left(\sigma_P\right)$ as a
union of orbits under the action of~$\Gal(\Kbar/K)$.  In section
\ref{sec:mag} we use similar ideas to give examples of EDS arising
from points on nonsplit elliptic curves which have only finitely many
irreducible terms.
\end{remark}

\section{Proof of Theorem~\ref{th:EDSprimitivedivisors}---Primitive Divisors in EDS}
\label{section:primitivedivisors}

In this section we prove a characteristic zero function field analogue
of the classical result~\cite{MR961918} that all but finitely many
terms in an elliptic divisibility sequence have a primitive divisor. 

Proving the existence of primitive valuations in EDS is much easier over function fields than
it is over number fields because there are no archimedean absolute
values. Over number fields, multiples~$nP$
of~$P$ will come arbitrarily close to~$\Ocal$ in the archimedean
metrics, necessitating the use of deep results from Diophantine
approximation. Over characteristic zero function fields, once some
multiple~$nP$ comes close to~$\Ocal$ in some~$v$-adic metric, no
multiple of~$P$ ever comes~$v$-adically closer to~$\Ocal$;
cf.\ Lemma~\ref{lemma:edsformalgp} below.

Inquiry into the number field analogues of the results in this section
has been motivated by the parallel question for Lucas sequences, answered
definitively by Bilu, Hanrot, and Voutier~\cite{MR1863855}. 
It is therefore natural to ask if
one can prove similar results for Lucas sequences over
function fields.  
Along these lines, Flatters and
Ward~\cite{flattersward} have shown that,
 for a polynomial ring over any field, all terms beyond the second
with indices coprime to the characteristic have a primitive valuation.

\subsection{Minimal proper regular models}

We begin with the somewhat technical
definition of a minimal proper regular model,
immediately followed by equivalent definitions
and properties
that may be more suitable for thinking about
elliptic divisibility sequences.

\begin{definition}
Let $C$ be a smooth projective
geometrically irreducible curve
over a number field~$K$,
and $E/K(C)$ an elliptic curve. 
A proper regular model for $E/K(C)$
is a pair $(\Ecal, \pi)$ consisting
of a regular scheme $\Ecal$ 
and a proper flat morphism
$\pi : \Ecal\rightarrow C$ 
with its generic fiber identified with~$E$. 

A proper regular model is $\emph{minimal}$ if,
given any other proper regular model $(\Ecal', \pi')$,
the birational map
$f:\Ecal'\dashrightarrow \Ecal$ satisfying $\pi\circ f = \pi'$
induced by the identification of the generic fibers,
is a morphism. 
\end{definition}

For any elliptic curve~$E/K(C)$, there is a unique
minimal proper regular model (\cite[IV.4.5]{MR1312368}),
and it is projective over~$K$. 
In particular, the minimal proper regular model is
an elliptic surface according to the definition of \cite[III]{MR1312368}.

If $E$ is defined over~$K$, then we can simply take $\Ecal = E\times C$,
as we claimed in Section~\ref{section:irredtermsineds}.

The following lemma shows how to determine the terms
in an EDS without computing a minimal proper regular model.
\begin{lemma}\label{lemma:minimalweierstrass}
Let $v$ be a valuation of $K(C)$, and
write $E$ in terms of a {minimal Weierstrass equation} at~$v$
{\textup (}as in \cite[VII]{MR2514094}{\textup )}. 
Then we have $$v(D_{nP}) = \max\{0,-\frac{1}{2} v(x([n]P))\}.$$
\end{lemma}
\begin{proof}
The minimal proper regular model $(\Ecal ,\pi )$ of $E/K(C)$ may have singular fibers. The zero section intersects fibers of $(\Ecal ,\pi )$ only at non-singular points. Since we are only interested in the pull-back
of the image of the zero section~$\Ocal$ by some other
section~$\s:C\to\Ecal$,
we only need to consider the 
identity component of the smooth part of each fiber. 
But the identity component of the smooth part of a fiber  is
given by the minimal Weierstrass equation (\cite[IV.6.1 and IV.9.1]{MR1312368}).
\end{proof}

A Weierstrass equation over $K(C)$ is minimal at all but finitely many valuations. 
For those valuations where it is not minimal, a change of coordinates makes
the Weierstrass equation minimal, which changes $v(D_{nP})$ by an amount
bounded independently of~$n$ (but
depending on the chosen Weierstrass equation).

\begin{example}
\label{example:EDSnotxdenom}
We illustrate Lemma~\ref{lemma:minimalweierstrass}.  Take~$\Ecal$ to be a
minimal proper regular model. 
Fix a minimal Weierstrass equation for~$E$ over some
affine piece of~$C$ and write~$P=(x_P,y_P)$. 
Then 
$2D_{nP}$ is close to the polar divisor of the function~$x_P\in
K(C)$, but may differ at valuations of $K(C)$ where the coefficients
are not regular or the discriminant is not invertible.

For example,
consider the curve and point
\[
  E:y^2=x^3-T^2x+1,\qquad P=(x_P,y_P)=(T,1)\in E\bigl(K(T)\bigr),
\]
over the rational function field~$K(T)$. It is minimal at all finite values for~$T$,
but to compute~$\s_P^*\Ocal$ at
$T=\infty$, we must change variables, say~$(x,y)=(T^2X,T^3Y)$. The new
equation is
\[
  E:Y^2=X^3-U^2X+U^6,
\]
with $U=T^{-1}$,
and the point~$P$ has coordinates~$(X_P,Y_P)=(U,U^3)$. 
This Weierstrass model is not smooth at $U=0$
(not even as a surface over $K$), so to find a regular model, we would have to blow up 
the singularity.
However, the discriminant $16U^6 (4-27U^6)$ is not divisible by $U^{12}$, hence
this is a minimal Weierstrass equation at $U=0$, so
so Lemma~\ref{lemma:minimalweierstrass} applies.  Since
\[ -\frac{1}{2}\ord_{U=0} X_P = -\frac{1}{2} \ord_{U=0} U = -\frac{1}{2} < 0,\]
we obtain
\[
  \ord_\infty D_P = \ord_\infty \s_P^*\Ocal = 0,\]
 in spite of having
 \[ -\frac{1}{2} \ord_\infty x_P =  - \frac{1}{2} \ord_\infty T = \frac{1}{2} > 0 \]
in the original model.
\end{example}

\subsection{Primitive valuations}

\begin{definition}
\label{def:primdiv}
Let~$K$ be a field, let~$C/K$ be a smooth projective curve, and let
$(D_n)_{n\ge1}$ be a sequence of effective divisors on~$C$. A
\emph{primitive valuation\footnote{To avoid confusion, we have changed
    terminology slightly and refer to primitive \emph{valuations},
    rather than primitive \emph{prime divisors}.  The reason that we
    do this is because the terms in our EDS are divisors on~$C$, and
    it is confusing to refer to divisors of divisors. Note that our
    ``prime divisors'' are points of~$C(\Kbar)$, which correspond to
    normalized valuations of the function field~$\Kbar(C)$.} of~$D_n$}
is a normalized valuation~$\g$ of~$\Kbar(C)$
(equivalently, a point~$\g\in C(\Kbar)$)
such that
\[
  \ord_\g(D_n)\ge1\qquad\text{and}\qquad \ord_\g(D_i)=0
  \quad\text{for all~$i<n$.}
\]
\end{definition}

We now begin the proof of Theorem~\ref{th:EDSprimitivedivisors}, which
we restate with a small amount of added notation.

\begin{theorem} 
\label{thm:primdiv}
Let~$K$ be a field of characteristic~$0$ and
let~$(D_{nP})_{n\ge1}$ be an elliptic divisibility sequence as in
Definition~$\ref{def:ellsurfeds}$.  Assume further that there is
no isomorphism $\psi:E\rightarrow E'$ over $\Kbar(C)$
with $E'$ defined over $\Kbar$ and $\psi(P)\in E'(\Kbar)$,
and that the point~$P\in
E\bigl(K(C)\bigr)$ is nontorsion.  Then there exists an~$N=N(E,P)$
such that for every~$n\ge N$, the divisor~$D_{nP}$ has a primitive
valuation.
\end{theorem}

To ease notation, we assume for the
remainder of this section that the constant field~$K$ is algebraically
closed, and of course we retain the assumption that
$\operatorname{char}(K)=0$.  Note that there is no loss of generality in this assumption, since we have adopted the convention of considering valuations on $\Kbar(C)$.

  We start with a standard lemma 
(cf.\ \cite[Lemma~4]{MR2081943}) whose conclusion over function fields
is much stronger than the analogous statement over number fields.  The term \emph{rigid divisibility} has been used for sequences with this strong property.  For 
the convenience of the reader, we include a proof via basic properties 
of the formal group.

\begin{lemma}
\label{lemma:edsformalgp}
Let~$(D_{nP})_{n\ge1}$ be an EDS associated to an elliptic surface as
in Definition~$\ref{def:ellsurfeds}$, let~$\g\in C(K)$ be a point
appearing in the support of some divisor in the EDS, and let
\[
  m = \min \{n\ge1 :  \ord_\g D_{nP} \ge 1\}.
\]
Then
\[
  \ord_\g D_{nP} = \begin{cases}
     \ord_\g D_{mP} &\text{if~$m\mid n$,} \\
      0 &\text{if~$m\nmid n$.} \\
  \end{cases}
\]
\end{lemma}
\begin{proof}
Let 
\[
  E\bigl(K(C)\bigr)_{\g,r}
  = \bigl\{ P\in E\bigl(K(C)\bigr) : \ord_\g \s_P^*\Ocal \ge r\bigr\}
   \cup \{\Ocal\}.
\]
Then~$E\bigl(K(C)\bigr)_{\g,r}$ is a subgroup of~$E\bigl(K(C)\bigr)$,
and
\[
  \ord_\g D_{nP} = \max\{r\ge0 : nP\in E\bigl(K(C)\bigr)_{\g,r}\}.
\]
This all follows from standard properties of the formal group of~$E$
over the completion~$K(C)_\g$ of~$K(C)$ at the valuation~$\ord_\g$;
see~\cite[Chapter~IV]{MR2514094}. It also follows that there is an
isomorphism of additive groups
\[
  \frac{E\bigl(K(C)_\g\bigr)_{\g,r}}{E\bigl(K(C)_\g\bigr)_{\g,r+1}}
  \cong \frac{\mathcal{M}_\g^r}{\mathcal{M}_\g^{r+1}} \cong K
  \qquad\text{for all~$r\ge1$,}
\]
where we use the notation $\mathcal{M}_\g$
to denote $K(C)_\g$'s maximal ideal.
The quotient is torsion-free since $\operatorname{char}(K) = 0$. 
\par
Let~$d=\ord_\g D_{mP}$. By assumption, we have~$d\ge1$ and
\[
  mP \in E\bigl(K(C)\bigr)_{\g,d} \smallsetminus E\bigl(K(C)\bigr)_{\g,d+1}.
\]
Since the quotient is torsion-free,
it follows that 
every multiple also satisfies
\[
  mkP \in E\bigl(K(C)\bigr)_{\g,d} \smallsetminus E\bigl(K(C)\bigr)_{\g,d+1},
\]
so~$\ord_\g D_{mkP} = d = \ord_\g D_{mP}$. 
\par
Conversely, suppose that~$\ord_\g D_{nP}\ge1$. To ease notation, 
let~$e=\ord_\g D_{nP}$. Then
\[
  nP\in E\bigl(K(C)\bigr)_{\g,e}
  \quad\text{and}\quad
  mP\in E\bigl(K(C)\bigr)_{\g,d},
\]
so the fact that~$\bigl\{E\bigl(K(C)\bigr)_{\g,r}\bigr\}_{r\ge0}$ give
a filtration of subgroups of~$E\bigl(K(C)\bigr)$ implies that
\[
  \gcd(m,n)P\in E\bigl(K(C)\bigr)_{\g,\min(d,e)}.
\]
Hence
\[
  \ord_\g D_{\gcd(m,n)P} \ge \min(d,e) \ge 1,
\]
so by the minimality of~$m$ we have~$m\le\gcd(m,n)$. Therefore~$m\mid n$,
which completes the proof of the lemma.
\end{proof}

\begin{definition}
Let~$E/K(C)$ and~$\Ecal\to C$ be as in Definition~$\ref{def:ellsurfeds}$. 
The \emph{canonical height} of a point~$P\in E\bigl(K(C)\bigr)$ is the
quantity
\[
  \hhat_E(P) = \lim_{n\to\infty} \frac{\deg \s_{nP}^*\Ocal}{n^2}.
\]
(If~$nP=\Ocal$, we set~$\s_{nP}^*\Ocal=0$.)  
\end{definition}

\begin{prop}\label{prop:non-vanishing-height}
The limit defining the canonical height exists, and the
function~$\hhat_E : E\bigl(K(C)\bigr)\to[0,\infty)$ is a quadratic form
satisfying
\begin{equation}
  \label{eqn:canhteqht}
  \hhat_E(P) = \deg\s_P^*\Ocal + O_E(1)
  \quad\text{for all~$P\in E\bigl(K(C)\bigr)$.}
\end{equation}
\textup(The~$O_E(1)$ depends on~$E/K(C)$.\textup)

Next, assume that there is no isomorphism
$\psi:E\rightarrow E'$ with $E'$ defined over $K$ and
$\psi(P)\in E'(K)$. Then we have
\[
  \hhat_E(P)=0 \quad\Longleftrightarrow\quad 
  P\in E\bigl(K(C)\bigr)_{\textup{tors}}.
\]
\end{prop}
\begin{proof}
A proof is given in~\cite[III.4.3]{MR1312368}, except for the final equivalence
in the case where $E$ is isomorphic to a curve over~$K$.

So assume $E$ is given by a Weierstrass equation with coefficients in~$K$. 
The point $P$ is not in~$E(K)$, so~$P$ is not a torsion point. 
The point $P$ induces a map $\sigma_P : C\rightarrow E$. 
Since $P\notin E(K)$, the map $\sigma_P$ is not constant,
i.e. $\deg (\sigma_p )$ is strictly positive. 
We show $\hhat_E(P) = \deg(\sigma_P)$.

An equation with coefficients in~$K$ is automatically a minimal Weierstrass equation for
every valuation $v$ of~$K(C)$, so Lemma~\ref{lemma:minimalweierstrass}
tells us
\[\hhat_E(P) = \lim_{n\rightarrow\infty} n^{-2}\sum_{v}\max\{0,-\frac{1}{2}v(x([n]P))\}
= \lim_{n\rightarrow\infty} n^{-2} \deg x([n]P).\]
Here $\deg x([n]P)$ is the degree of the map $x([n]P) : C\rightarrow\mathbf{P}^1$,
and we have $x([n]P) = x\circ [n] \circ \sigma_P$. 
In particular, multiplicativity of degrees
tells us $\deg x([n]P) = 2n^2 \deg \sigma_P$.
\end{proof}
\begin{remark}
It is not hard to
derive explicit upper and lower bounds for the~$O_E(1)$ in
\eqref{eqn:canhteqht} in terms of geometric invariants of the elliptic
surface~$\Ecal$; see for example~\cite{MR1035944,MR0419455}.
\end{remark}

\begin{proof}[Proof of Theorem~$\ref{thm:primdiv}$]
The proof follows the lines of the proof over number fields;
cf.\ \cite{MR961918}. The point $P$ is not a torsion point. 
From Proposition~\ref{prop:non-vanishing-height} we know
$\hhat_E(P)>0$. Suppose that~$D_{nP}$ has no primitive
valuations. Then
\begin{align*}
  D_{nP}
  &= \sum_{\g\in C} \ord_\g(D_{nP})(\g) \\
  &\le \sum_{m<n} \sum_{\g\in\Supp(D_{mP})} \ord_\g(D_{nP})(\g) 
    &&\text{by assumption,} \\
  &\le \sum_{m\mid n, m<n} \sum_{\g\in\Supp(D_{mP})} \ord_\g(D_{mP})(\g) 
    &&\text{from Lemma~\ref{lemma:edsformalgp},} \\
  &= \sum_{m\mid n, m<n} D_{mP}.
\end{align*}
Taking degrees and using properties of the canonical height yields
\begin{align*}
  n^2\hhat_E(P)
  &= \hhat_E(nP) \\*
  &= \deg D_{nP} + O(1) \\
  &\le \sum_{m\mid n, m<n} \deg D_{mP} + O(1) \\
  &= \sum_{m\mid n, m<n} \bigl(\hhat_E(mP)+O(1)\bigr) \\
  &= \sum_{m\mid n, m<n} \bigl(m^2\hhat_E(P)+O(1)\bigr) \\
  &\le n^2 \biggr(\sum_{m\mid n,m>1} \frac{1}{m^2}\biggl)\hhat_E(P)
       + O(n) \\
  &< n^2 \bigl(\zeta(2)-1\bigr)\hhat_E(P) + O(n) \\*
  &< \frac23n^2\hhat_E(P) + O(n).
\end{align*}
Since~$\hhat_E(P)>0$, this gives an upper bound for~$n$.
\end{proof}

\begin{remark}
It is an interesting question to give an explicit upper bound for the
value of~$N(E,P)$ in Theorem~\ref{thm:primdiv}, i.e., for the
largest value of~$n$ such~$D_{nP}$ has no primitive valuation. 
Using the function field version of Lang's height lower bound
conjecture, proven in~\cite{MR948108}, and standard explicit estimates
for the difference between the Weil height and the canonical height,
it may be possible to prove that for EDS associated to a minimal
model, the bound~$N(E,P)$ may be chosen to depend only
on the genus of the function field~$K(C)$, independent of~$E$
and~$P$. However, the details are sufficiently intricate that we will
leave the argument for a subsequent note. 
(See~\cite{ingramsilverman06} for a weaker result over number
fields, conditional on the validity of Lang's height lower bound
conjecture for number fields.)
\end{remark}

\begin{remark}\label{rem:maarten}
Theorem~\ref{thm:primdiv} ensures, in the non-split case, that all but
finitely many terms in an EDS over a function field have a primitive
valuation.  If the base field~$K$ is a number field, then these
valuations correspond to divisors defined over~$K$, and thus
are attached to a Galois orbit of points.  It is
natural to ask about the degrees of these primitive valuations.  Note
that if~$\gamma\in C(\Kbar)$ is in the support of one of these
primitive valuations, then~$P$ specializes to a torsion point on
the fiber above~$\gamma$, and so it follows
from~\cite[Theorem~III.11.4]{MR1312368} (or elementary estimates if
the fiber is singular) that the height of~$\gamma$ is bounded by a
quantity depending only on~$E$.  One immediately obtains
an~$O(\log{n})$-lower bound on the degree of the smallest primitive
valuation of~$D_{nP}$.  Maarten Derickx has pointed out to the authors
that one can prove a weaker, but more uniform, lower bound using deep
results of Merel, Oesterl\'{e}, and Parent (see~\cite{MR1665681} and
the addendum to~\cite{MR1321648}).  In particular, one obtains a lower
bound which is logarithmic in the largest prime divisor of~$n$, with
constants depending only on the underlying number field, independent
of~$E$.
\end{remark}


\section{Magnification and Elliptic Divisibility Sequences}
\label{sec:mag}

As usual,  let~$C/K$ be a smooth projective curve defined over a
field~$K$ of characteristic zero and consider an elliptic
divisibility sequence $\left(D_{nP}\right)_{n\ge 1}$ arising as in
Definition~\ref{def:ellsurfeds} from a~$K(C)$-point~$P$ on an elliptic
curve~$E/K(C)$. 
Suppose that $E$ and $P$ satisfy the hypotheses
of Theorem~\ref{thm:primdiv}. 
That theorem then says that there exists a
sequence~$\gamma_1,\gamma_2,\gamma_3,\ldots$ of closed points of~$C$
such that
\[
  \ord_{\gamma_{n}}(D_{mP}) > 0 \quad\Longleftrightarrow\quad  n\mid m.
\]
Theorem~\ref{th:mainEDSresult} provides examples of elliptic
divisibility sequences such that for infinitely many indices~$n$, the
support of~$D_{nP}$ is exactly the~$\Gal(\Kbar/K)$-orbit of
the single point~$\gamma_{n}$. 

However, the example of Lucas sequences with finitely many irreducible
terms~\eqref{eqn:genfac} suggests that the same should be true for
some EDS.  In this section we describe properties of EDS that ensure
that for all sufficiently large~$n$, the divisor~$D_{nP}$ contains at
least two distinct Galois orbits. 

\begin{definition}
\label{def:magnification}
An elliptic divisibility sequence~$\left( D_{nP}\right)_{n\ge 1}$ attached to
an elliptic curve~$E/K(C)$ is said to be \emph{magnified} over $K(C)$ if there is an elliptic curve~$E'/K(C)$, 
an isogeny~$\tau:E'\longrightarrow E$ defined over~$K(C)$ that is not an isomorphism, and a point~$P'~\in~E'\bigl( K(C)\bigr)$ such that~$P = \tau(P')$.
\end{definition}

The following result is a variant of \cite[Theorem~1.5]{MR2429645}. 

\begin{theorem}
\label{thm:two-primitive-divisors}
Assume that~$E$ and $P$ satisfy the hypotheses
of Theorem~\ref{thm:primdiv}, and that~$\left( D_{nP}\right)_{n\ge 1}$ is
magnified over $K(C)$.  Then there is a constant~$M = M (E,P)$ such that for every
index~$n>M$, the support of the divisor~$D_{nP}$ includes at least two valuations that are not~$\Gal\left(\Kbar/K\right)$-conjugates
of one another.
\end{theorem}

\begin{proof}
Let~$\tau : E'\longrightarrow E$ and~$P'\in E'\bigl(K(C)\bigr)$ be
defined as in Definition~\ref{def:magnification}, and let
$\left(D_{nP'}\right)_{n\ge 1}$ be the elliptic divisibility sequence
associated to~$P'$. 

The isogeny $\tau$ induces a morphism $\tau$ from the N\'eron model of $E'$ to the N\'eron model of $E$. The zero section intersects fibers of the minimal proper regular model only at non-singular points, so we know from the relationship between the minimal proper regular model and the N\'eron model \cite[IV.6.1 and IV.9.1]{MR1312368} that, for any index~$n$, the divisor
\begin{equation}\label{equations-effectivity-DnP-DnP'}
  D_{nP} -D_{nP'} = \sigma_{nP}^*(\Ocal_{E}) - \sigma_{nP'}^*(\Ocal_{E'})
  = \sigma_{nP'}^*\bigl(\tau^*(\Ocal_{E})-\Ocal_{E'}\bigr)
\end{equation}
is effective (cf.\ \cite[Lemma 2.13]{MR2377368} for a complete proof of the analogous result for elliptic divisibility sequences defined over number fields).

We required the hypotheses of Theorem~\ref{thm:primdiv}
only for $(E,P)$, but the proof of that
theorem holds for $(E',P')$ as well.
Indeed, the hypotheses are used in the proof
of Theorem~\ref{thm:primdiv} only for showing~$\hhat(P)>0$,
which implies $\hhat(P')>0$ via~$\tau$.
In particular, there is a bound
$N(E',P')$ such that for every~$n~>~N(E',P')$, the divisor~$D_{nP'}$
has a primitive valuation, say~$\gamma_n'~\in~C(\Kbar)$. Then~$\gamma_n'$ occurs also in the support of~$D_{nP}$.
Further, since every divisor \text{$D_{mP'}\in\Div(C)$}
is defined over~$K$, we see that every
$\Gal\left(\Kbar/K\right)$-conjugate of a primitive valuation of
$D_{nP'}$ is again a primitive valuation of~$D_{nP'}$.
Hence
Theorem~\ref{thm:two-primitive-divisors} is proven once we show
that for all sufficiently large~$n$, the support of~$D_{nP}$ contains a valuation~$\gamma_{n}\in C(\Kbar)$ with \text{$\ord_{\gamma_{n}}(D_{nP'}) =
0$}. We do this by modifying the proof of Theorem~\ref{thm:primdiv}.

Suppose that~$n$ is an index such that
$\ord_{\gamma}(D_{nP'}) > 0$ for \emph{every} valuation~$\gamma$ belonging to the support of~$D_{nP}$.  We will show that~$n$ is bounded.
Let~$d~=~\deg(\tau)\ge2$.
Applying \eqref{equations-effectivity-DnP-DnP'} and its analogue
for the dual of~$\tau$, we get
\[
  \ord_{\gamma}(D_{ndP'})\ge\ord_{\gamma}(D_{nP})\ge\ord_{\gamma}(D_{nP'})
\]
for every valuation~$\gamma$.
If $\gamma$ belongs to the support of~$D_{nP}$, then
by assumption we also have $\ord_{\gamma}(D_{nP'})>0$,
so Lemma~\ref{lemma:edsformalgp} tells us that the outermost orders are equal.
In particular, we get
\[
  \ord_{\gamma}(D_{nP})  =\ord_{\gamma}(D_{nP'}),
\]
which is also true if $\gamma$ does not belong to the
support of~$D_{nP}$.
It follows that $D_{nP} = D_{nP'}$. Taking degrees, this implies
\begin{align*}
  n^{2}d\hhat (P') = \hhat (nP) &\le \deg D_{nP} + O(1) = \deg D_{nP'} + O(1)\\
  &\le \hhat (nP') + O(1)
  \le  n^{2}\hhat (P') + O(1).
\end{align*}
In particular, the index~$n$ is bounded since $d>1$.
\end{proof}

\begin{remark}
The proof of Theorem~\ref{thm:two-primitive-divisors} is based on the
effectiveness of the divisor~$D_{nP} -D_{nP'}$. 
Corrales-Rodrig\'{a}\~{n}ez and
Schoof~\cite{Corrales-Rodriganez-Schoof} proved that, in number
fields, the analog to the magnification condition is the only way to
construct a pair of elliptic divisibility sequences
$\left(B_{n}\right)_{n\ge 1}$ and~$\left(D_{n}\right)_{n\ge 1}$ such
that~$B_{n}\mid D_{n}$ for every~$n\ge 1$. 
\end{remark}

\begin{remark}
Theorem~\ref{thm:two-primitive-divisors} implies that
Theorem~\ref{th:mainEDSresult} cannot be generalized to magnified
points.
\end{remark}


\section{Examples}
\label{sec:ex}

In this section we provide examples of Lucas sequences and elliptic
divisibility sequences over function fields that illustrate some of
our results.  Computations were performed with Sage Mathematics Software \cite{Sage}.

\subsection{Lucas sequences over~$K[T]$}
We provide some examples illustrating the two cases of
Remark~\ref{rem:lucaskt}.  If~$f(T)\in K[T]$ has prime degree and
$f(T)-1$ is irreducible, then the Lucas sequence
\[
  L_n=\frac{f(T)^n-1}{f(T)-1}
\]
is amenable.  Lemma \ref{lem:multlemma} then tells us that~$L_q$ is
irreducible for all sufficiently large~$q$ such that~$f(T)$ is
irreducible modulo some~$\fq \mid q$.  Looking at the proof of Lemma
\ref{lem:multlemma}, we see that the following notion of
``sufficiently large'' suffices.
\begin{enumerate}
  \item~$f(T)$ has~$\fq$-integral coefficients, and leading
    coefficient a~$\fq$-unit.
  \item~$\QQ(\zeta_q)$ is linearly disjoint from~$K$.
\end{enumerate}
For example,~$f(T) = T^2 + 1 \in \QQ[T]$ is irreducible modulo all
primes \text{$q\equiv 3\MOD{4}$}.  Hence in the Lucas sequence
\[
  L_n=\frac{(T^2+2)^n-1}{T^2+1},
\] 
the term~$L_q$ is irreducible for all primes~\text{$q \equiv 3
  \MOD{4}$}.  In fact, we checked that~$L_q$ is irreducible for all
primes $q\le 1009$, which suggests that~$L_q$ may be irreducible
for all primes. The first few terms, in factored form, are:
{\small\begin{align*}
  L_1 &= 1, \\
  L_2 &= T^2 + 3,\\
  L_3 &= T^4 + 5T^2 + 7, \\
  L_4 &= (T^2 + 3)(T^4 + 4T^2 + 5), \\
  L_5 &= T^8 + 9T^6 +31T^4 + 49T^2 + 31, \\
  L_6 &= (T^2 + 3)(T^4 + 3T^2 + 3)(T^4 + 5T^2 + 7), \\
  L_7 &= T^{12} + 13T^{10} + 71T^8 + 209T^6 + 351T^4 + 321T^2 + 127, \\
  L_8 &= (T^2 +3)(T^4 + 4T^2 + 5)(T^8 + 8T^6 + 24T^4 + 32T^2 + 17), \\
  L_9 &= (T^4 + 5T^2+ 7)(T^{12} + 12T^{10} + 60T^8 + 161T^6 + 246T^4 + 204T^2 + 73), \\
  L_{10} &= (T^2 + 3)(T^8 + 7T^6 + 19T^4 + 23T^2 + 11)(T^8 + 9T^6 + 31T^4 +49T^2 + 31).
\end{align*}}

In general, the Chebotarev density theorem used in
Lemma~\ref{lem:EDSchebotarev} provides us with a specific value for
the lower density.  In the case that the extension of~$K$ generated by
a root of~$f(T)$ is Galois of prime degree~$p$, the lower density
provided by our proof is~$(p-1)/p$. 

For a concrete example of the second type of Lucas sequence described
in Remark \ref{rem:lucaskt}, we consider
\[
  L_n=\frac{f^n - g^n}{f - g}\in \ZZ[T],
\]
where
\[
  f = T+S,\quad g =T-S,\quad S^2=T^3-2.
\]
The first few terms of this sequence are
{\small\begin{align*}
  L_1 &= 1, \\
  L_2 & = 2T, \\
  L_3 &= (T+1)(T^2+2T-2), \\
  L_4 &= 4T(T-1)(T^2+2T+2),\\
  L_5 &= T^6+10T^5+5T^4-4T^3-20T^2+4,\\
  L_6 &= 2T(T+1)(T^2+2T-2)(3T^3+T^2-6),\\ 
  L_7 &= T^9+21T^8+35T^7+T^6-84T^5-70T^4+12T^3+84T^2-8. 
\end{align*}}
We have checked that~$L_q$ is irreducible for all primes~$5\leq q\leq
1009$, but we note that~$L_q$ is reducible for~$q=3$.  It seems likely
that all but finitely many prime-indexed terms of this sequence are
irreducible, but this sequence illustrates the fact that amenability
does not imply that \emph{every} prime-indexed term is irreducible.

\subsection{Split elliptic divisibility sequences}

Let
\[
  E: y^2 +a_1 xy + a_3 y = x^3 + a_2x^2 + a_4 x + a_6
\]
be an elliptic curve defined over~$K$.  Then for any curve~$C/K$, we
may consider~$E$ as a split elliptic curve over the function
field~$K(C)$.

We now take~$C = E$ and consider~$E$ as an elliptic curve over its own
function field~$K(E)=K(x,y)$.  Then~$D_{nP}$ for~$P = (x,y)$ is
essentially the divisor of the division polynomial $\Psi_n(x,y)$.
This constitutes a universal example in the following sense.
Suppose~$C$ is a curve defined over~$K$ with a rational map~$C
\rightarrow E$.  Then, considering~$E$ as a curve over~$K(E)$, pulling
back by this map gives~$E$ as a curve over~$K(C)$:
\[
  \xymatrix{
    E_{/K(C)} \ar[r] \ar[d] & E_{/K(E)} \ar[d] \\
    \Spec K(C) \ar[r] & \Spec K(E) \\
  }
\]
Pulling back the point~$P = (x,y)$ across the top gives rise to a
$K(C)$-point on~$E$.  Conversely, any~$K(C)$-point on~$E$ gives rise
to a map~$C \rightarrow E$.  In particular, the only~$K(T)$-points of
$E$ are its~$K$-points, since the only maps~$\PP^1 \rightarrow E$ are
constant.

To illustrate this construction, suppose that
\[
  E: y^2 = x^3 - 7x + 6.
\]
Consider the curve
\[
  C: v^2 = u^3 - 7(u^3 + 2)^4u + 6(u^3 +2)^6
\]
and the map 
\[
  C \longrightarrow E,\qquad (u,v) \longmapsto (u/(u^3+2)^2,v/(u^3+2)^3).  
\]
Then
\[
  P = (u/(u^3+2)^2, v/(u^3+2)^3) \in E\bigl(K(C)\bigr),
\]
and the associated sequence of~$D_{nP}$ (in factored form, where we
identify~$D_Q$ with a function on~$C$ whose divisor is
$D_Q-\deg(D_Q)(\Ocal)$) begins
{\small\begin{align*}
  D_P &= (u^3 + 2), \\
  D_{2P} &= 2y(u^3 + 2), \\
  D_{3P} &= (u^3 + 2) (72u^{22} + 1008u^{19} + 5964u^{16} + 19320u^{13} -
49u^{12} \\& + 36960u^{10} - 392u^9 + 42u^8 + 41676u^7 - 1176u^6 + 168u^5
\\ & + 25551u^4 - 1568u^3 + 168u^2 + 6528u - 784),\\
  D_{4P} &= 4y(u^3 + 2)(288u^{42} + 8064u^{39} + 104160u^{36} + 822528u^{33} \\& +
4435592u^{30} + 504u^{28} + 17275648u^{27} + 9072u^{25} \\& + 50100936u^{24} +
71988u^{22} + 109870016u^{21} + 330456u^{19} \\& + 183006341u^{18} + 966672u^{16}
+ 230282052u^{15} + 441u^{14}\\& + 1867572u^{13} + 215342212u^{12} + 3528u^{11}
+ 2380539u^{10}\\& + 144988252u^9 + 10584u^8 + 1927548u^7 + 66365219u^6
\\&+ 14112u^5 + 897708u^4 + 18454080u^3 + 7056u^2 \\&+ 182784u +
2345536).
\end{align*}}
We also computed~$D_{5P} - D_P$, which has degree~$84$ and is
irreducible (as a polynomial in $u$).

\subsection{An isogeny}
As an example to which Theorem \ref{thm:isogenyEDS} applies, consider
the elliptic curves
\begin{align*}
  E&: y^2 + y = x^3 - x^2 - 10x - 20, \\
  C&: v^2 + v = u^3 - u^2 - 7820u - 263580.
\end{align*}
There is an isogeny~$\sigma_P: C \rightarrow E$ of degree~$5$ such 
that the divisor
\[
  \sum_{Q \in \ker(\sigma_P)} (Q) - (\Ocal)
\]
is irreducible over~$\QQ$.  The map~$\sigma_P$ gives a point~$P$ on
$E$ as a curve over~$K(C)$.  We find that, in factored form,
{\small\begin{align*}
   D_P &= (5u^2 + 505u + 12751) \\
  D_{3P} &= (5u^2 + 505u + 12751)(3u^4 - 4u^3 - 46920u^2 - 3162957u \\&
   - 60098081) (u^{16} + 808u^{15}  + 307664u^{14} + 73114536u^{13} \\ & +
  12109319702u^{12}  + 1478712412670u^{11} + 137408300375962u^{10} \\& +
  9888567316290696u^9 + 555597255218203792u^8 \\& + 24384290372532564144u^7
  + 830287549319036362345u^6 \\& + 21602949256698317741635u^5 +
  418237794866116560977925u^4 \\& + 5763041398838852610101023u^3 +
  52312834246514003927525299u^2 \\& + 268864495959470526718080718u +
  530677345945019287998317531).
\end{align*}}
The factor
\[
  3u^4 - 4u^3 - 46920u^2 - 3162957u  - 60098081
\]
is the third division polynomial for~$C$, as expected from the proof
Theorem~\ref{thm:isogenyEDS}.


\begin{thebibliography}{10}

\bibitem{MR1185022}
M.~Ayad.
\newblock Points {$S$}-entiers des courbes elliptiques.
\newblock {\em manuscripta math.},
76(3--4):305--324, 1992.

\bibitem{Bang}
A.~S.~Bang.
\newblock Taltheoretiske unders{\o}lgelser.
\newblock {\em Tidskrift f.~Math.}, 5:70--80 and 130--137, 1886.

\bibitem{MR1863855}
Yu.~Bilu, G.~Hanrot, and P.~M.~Voutier.
\newblock Existence of primitive divisors of {L}ucas and {L}ehmer numbers.
\newblock {\em J.~Reine Angew.~Math.}, 539:75--122, 2001.
\newblock With an appendix by M.~Mignotte.

\bibitem{MR1502458}
R.~D.~Carmichael.
\newblock On the numerical factors of the arithmetic forms
  {$\alpha^n\pm\beta^n$}.
\newblock {\em Ann.~of Math.~(2)}, 15(1-4):30--70, 1913/14.

\bibitem{MR0866702}
D.~V.~Chudnovsky and G.~V.~Chudnovsky.
\newblock Sequences of numbers generated by addition in formal groups and new
  primality and factorization tests.
\newblock {\em Adv.~in Appl.~Math.}, 7(4):385--434, 1986.

\bibitem{MR2377127}
G.~Cornelissen and K.~Zahidi.
\newblock Elliptic divisibility sequences and undecidable problems about
  rational points.
\newblock {\em J.~Reine Angew.~Math.}, 613:1--33, 2007.

\bibitem{Corrales-Rodriganez-Schoof}
C.~Corrales-Rodrig\'a$\tilde{n}$ez and R.~Schoof.
\newblock The support problem and its elliptic analogue.
\newblock {\em Journal of Number Theory}, 64(2):276--290, 1997.

\bibitem{MR1484896}
H.~Dubner and W.~Keller.
\newblock New {F}ibonacci and {L}ucas primes.
\newblock {\em Math.~Comp.}, 68(225):417--427, S1--S12, 1999.

\bibitem{MR1321648}
B.~Edixhoven.
\newblock Rational torsion points on elliptic curves over number fields (after
  {K}amienny and {M}azur).
\newblock {\em Ast\'erisque}, (227):Exp.\ No.\ 782, 4, 209--227, 1995.
\newblock S{\'e}minaire Bourbaki, Vol.~1993/94.

\bibitem{MR1815962}
M.~Einsiedler, G.~Everest, and T.~Ward.
\newblock Primes in elliptic divisibility sequences.
\newblock {\em LMS J.~Comput.~Math.}, 4:1--13 (electronic), 2001.

\bibitem{MR2480276}
K.~Eisentr{\"a}ger and G.~Everest.
\newblock Descent on elliptic curves and {H}ilbert's tenth problem.
\newblock {\em Proc.~Amer.~Math.~Soc.}, 137(6):1951--1959, 2009.

\bibitem{MR1144318}
N.~D.~Elkies.
\newblock Distribution of supersingular primes.
\newblock {\em Ast\'erisque}, (198-200):127--132 (1992), 1991.

\bibitem{MR2429645}
G.~Everest, P.~Ingram, V.~Mah{\'e}, and S.~Stevens.
\newblock The uniform primality conjecture for elliptic curves.
\newblock {\em Acta Arith.}, 134(2):157--181, 2008.

\bibitem{MR2164113}
G.~Everest and H.~King.
\newblock Prime powers in elliptic divisibility sequences.
\newblock {\em Math.~Comp.}, 74(252):2061--2071 (electronic), 2005.

\bibitem{MR2220263}
G.~Everest, G.~Mclaren, and T.~Ward.
\newblock Primitive divisors of elliptic divisibility sequences.
\newblock {\em J.~Number Theory}, 118(1):71--89, 2006.

\bibitem{MR2045409}
G.~Everest, V.~Miller, and N.~Stephens.
\newblock Primes generated by elliptic curves.
\newblock {\em Proc.~Amer.~Math.~Soc.}, 132(4):955--963 (electronic), 2004.

\bibitem{MR1961589}
G.~Everest and T.~Ward.
\newblock Primes in divisibility sequences.
\newblock {\em Cubo Mat.~Educ.}, 3(2):245--259, 2001.

\bibitem{flattersward}
A.~Flatters and T.~Ward.
\newblock Polynomial {Z}sigmondy theorems, 2010.
\newblock \url{arXiv:1002.4829}.

\bibitem{MR948108}
M.~Hindry and J.~H.~Silverman.
\newblock The canonical height and integral points on elliptic curves.
\newblock {\em Invent.~Math.}, 93(2):419--450, 1988.

\bibitem{MR2301226}
P.~Ingram.
\newblock Elliptic divisibility sequences over certain curves.
\newblock {\em J.~Number Theory}, 123(2):473--486, 2007.

\bibitem{MR2605536}
P.~Ingram.
\newblock A quantitative primitive divisor result for points on elliptic
  curves.
\newblock {\em J.~Th\'eor.~Nombres Bordeaux}, 21(3):609--634, 2009.

\bibitem{ingramsilverman06}
P.~Ingram and J.~H.~Silverman.
\newblock Uniform estimates for primitive divisors in elliptic divisibility
  sequences.
\newblock In {\em Number theory, Analysis and Geometry (In memory of Serge
  Lang)}, pages 233--263.~Springer-Verlag, 2011.

\bibitem{StangeLauter09}
K.~E.~Lauter and K.~E.~Stange.
\newblock The elliptic curve discrete logarithm problem and equivalent hard
  problems for elliptic divisibility sequences.
\newblock In {\em Selected Areas in Cryptography 2008}, volume 5381 of {\em
  Lecture Notes in Comput.~Sci.}, pages 309--327. Springer, Berlin, 2009.

\bibitem{MR2196797}
F.~Luca and P.~St{\u{a}}nic{\u{a}}.
\newblock Prime divisors of {L}ucas sequences and a conjecture of {S}ka\l ba.
\newblock {\em Int.~J.~Number Theory}, 1(4):583--591, 2005.

\bibitem{Mahe-Explicit-bounds}
V.~Mah\'{e}.
\newblock Prime power terms in elliptic divisibility sequences.
\newblock preprint, january 2010.

\bibitem{GIMPS}
Mersenne~Research~Inc.
\newblock Great internet mersenne prime search.
\newblock \url{http://mersenne.org}.

\bibitem{neukirch}
J.~Neukirch
\newblock {\em Algebraic number theory}, volume 322 of {\em Grundlehren der Mathematischen Wissenschaften [Fundamental Principles of Mathematical Sciences]}.
\newblock Springer-Verlag, Berlin, 1999.

\bibitem{MR1665681}
P.~Parent.
\newblock Bornes effectives pour la torsion des courbes elliptiques sur les
  corps de nombres.
\newblock {\em J.~Reine Angew.~Math.}, 506:85--116, 1999.

\bibitem{MR1992832}
B.~Poonen.
\newblock Hilbert's tenth problem and {M}azur's conjecture for large subrings
  of {$\mathbb{Q}$}.
\newblock {\em J.~Amer.~Math.~Soc.}, 16(4):981--990 (electronic), 2003.

\bibitem{reynolds}
J.~Reynolds.
\newblock On the pre-image of a point under an isogeny and siegel's theorem.
\newblock {\em New York Journal of Mathematics}, 17:163--172, 2011.

\bibitem{MR0344221}
A.~Schinzel.
\newblock Primitive divisors of the expression {$A\sp{n}-B\sp{n}$} in algebraic
  number fields.
\newblock {\em J.~Reine Angew.~Math.}, 268/269:27--33, 1974.
\newblock Collection of articles dedicated to Helmut Hasse on his seventy-fifth
  birthday, II.

\bibitem{MR0155816}
I.~Seres.
\newblock \"{U}ber die {I}rreduzibilit\"at gewisser {P}olynome.
\newblock {\em Acta Arith.}, 8:321--341, 1962/1963.

\bibitem{MR0179158}
I.~Seres.
\newblock Irreducibility of polynomials.
\newblock {\em J.~Algebra}, 2:283--286, 1965.

\bibitem{MR0387283}
J.-P.~Serre.
\newblock Propri\'et\'es galoisiennes des points d'ordre fini des courbes
  elliptiques.
\newblock {\em Invent.~Math.}, 15(4):259--331, 1972.

\bibitem{MR644559}
J.-P.~Serre.
\newblock Quelques applications du th\'eor\`eme de densit\'e de {C}hebotarev.
\newblock {\em Inst.~Hautes \'Etudes Sci.~Publ.~Math.}, (54):323--401, 1981.

\bibitem{MR1484415}
J.-P.~Serre.
\newblock {\em Abelian {$l$}-adic representations and elliptic curves},
  volume~7 of {\em Research Notes in Mathematics}.
\newblock A K Peters Ltd., Wellesley, MA, 1998.

\bibitem{Shipsey00}
R.~Shipsey.
\newblock {\em Elliptic divisibility sequences}.
\newblock PhD thesis, Goldsmith's College (University of London), 2000.

\bibitem{MR961918}
J.~H.~Silverman.
\newblock Wieferich's criterion and the {$abc$}-conjecture.
\newblock {\em J.~Number Theory}, 30(2):226--237, 1988.

\bibitem{MR1035944}
J.~H.~Silverman.
\newblock The difference between the {W}eil height and the canonical height on
  elliptic curves.
\newblock {\em Math.~Comp.}, 55(192):723--743, 1990.

\bibitem{MR1312368}
J.~H.~Silverman.
\newblock {\em Advanced topics in the arithmetic of elliptic curves}, volume
  151 of {\em Graduate Texts in Mathematics}.
\newblock Springer-Verlag, New York, 1994.

\bibitem{MR2081943}
J.~H.~Silverman.
\newblock Common divisors of elliptic divisibility sequences over function
  fields.
\newblock {\em Manuscripta Math.}, 114(4):431--446, 2004.

\bibitem{MR2178070}
J.~H.~Silverman.
\newblock {$p$}-adic properties of division polynomials and elliptic
  divisibility sequences.
\newblock {\em Math.~Ann.}, 332(2):443--471 (Addendum 473--474), 2005.

\bibitem{MR2514094}
J.~H.~Silverman.
\newblock {\em The arithmetic of elliptic curves}, volume 106 of {\em Graduate
  Texts in Mathematics}.
\newblock Springer, Dordrecht, second edition, 2009.

\bibitem{MR2747036}
J.~H.~Silverman and K.~E.~Stange.
\newblock Terms in elliptic divisibility sequences divisible by their indices.
\newblock {\em Acta Arith.}, 146(4):355--378, 2011.

\bibitem{MR2226354}
J.~H.~Silverman and N.~Stephens.
\newblock The sign of an elliptic divisibility sequence.
\newblock {\em J.~Ramanujan Math.~Soc.}, 21(1):1--17, 2006.

\bibitem{Stange10}
K.~E.~Stange.
\newblock Elliptic nets and elliptic curves.
\newblock preprint, April 2010.

\bibitem{MR2423649}
K.~E.~Stange.
\newblock The {T}ate pairing via elliptic nets.
\newblock In {\em Pairing-based cryptography---{P}airing 2007}, volume 4575 of
  {\em Lecture Notes in Comput.~Sci.}, pages 329--348. Springer, Berlin, 2007.

\bibitem{Sage}
W.\thinspace{}A.~Stein et~al.
\newblock {\em {S}age {M}athematics {S}oftware ({V}ersion 4.6.2)}.
\newblock The Sage Development Team, 2011.
\newblock {\tt http://www.sagemath.org}.

\bibitem{MR0476628}
C.~L.~Stewart.
\newblock Primitive divisors of {L}ucas and {L}ehmer numbers.
\newblock In {\em Transcendence theory: advances and applications ({P}roc.
  {C}onf., {U}niv.~{C}ambridge, {C}ambridge, 1976)}, pages 79--92.~Academic
  Press, London, 1977.

\bibitem{MR2377368}
M.~Streng.
\newblock Divisibility sequences for elliptic curves with complex
  multiplication.
\newblock {\em Algebra Number Theory}, 2(2):183--208, 2008.

\bibitem{voutieryabuta}
P.~M.~Voutier and M.~Yabuta.
\newblock Primitive divisors of certain elliptic divisibility sequences, 2010.
\newblock \url{arXiv:1009.0872}.

\bibitem{MR1284673}
P.~M.~Voutier.
\newblock Primitive divisors of {L}ucas and {L}ehmer sequences.
\newblock {\em Math.~Comp.}, 64(210):869--888, 1995.

\bibitem{MR679454}
S.~S.~Wagstaff, Jr.
\newblock Divisors of {M}ersenne numbers.
\newblock {\em Math.~Comp.}, 40(161):385--397, 1983.

\bibitem{MR0027286}
M.~Ward.
\newblock The law of repetition of primes in an elliptic divisibility sequence.
\newblock {\em Duke Math.~J.}, 15:941--946, 1948.

\bibitem{MR0023275}
M.~Ward.
\newblock Memoir on elliptic divisibility sequences.
\newblock {\em Amer.~J.~Math.}, 70:31--74, 1948.

\bibitem{watson}
G.~N.~Watson.
\newblock The problem of the square pyramid.
\newblock {\em Messenger of Math.}, 48:1--22, 1918.

\bibitem{MR0419455}
H.~G.~Zimmer.
\newblock On the difference of the {W}eil height and the {N}\'eron-{T}ate
  height.
\newblock {\em Math.~Z.}, 147(1):35--51, 1976.

\bibitem{MR1546236}
K.~Zsigmondy.
\newblock Zur {T}heorie der {P}otenzreste.
\newblock {\em Monatsh.~Math.~Phys.}, 3(1):265--284, 1892.

\end{thebibliography}

\end{document}